\documentclass{amsart}
\usepackage{etex}
\reserveinserts{28}
\usepackage{xypic}
\usepackage{tikz}
\usepackage{tikz-cd}
\usepackage{todonotes}
\usetikzlibrary{matrix,arrows,decorations.pathmorphing}
\usepackage{MnSymbol}
\usepackage{latexsym}
\usepackage{pgf}
\usepackage{mathtools}
\usepackage{stmaryrd}
\usepackage{pictex}
\usepackage{amsmath,amscd}
\usepackage{afterpage}
\usepackage{youngtab}
\usepackage{psfrag}
\usepackage[boxsize=.25 em]{ytableau}
\usepackage{verbatim}
\usepackage{graphics}

\newtheorem{Thm}{Theorem}[section]
\newtheorem{theorem}[Thm]{Theorem}
\newtheorem{lemma}[Thm]{Lemma}

\newtheorem{proposition}[Thm]{Proposition}
\newtheorem{example}[Thm]{Example}
\newtheorem{definition}[Thm]{Definition}

\usepackage{pict2e}

\newcommand{\Z}{\mathbb{Z}}

\newcommand{\kk}{\mathbf{k}}

\newcommand{\mbf}{\mathbf}

\newcommand{\mc}[1]{\mathcal{#1}}
\newcommand{\tb}[1]{\textbf{#1}}

\newcommand{\Osq}{\mathbin{\text{\scalebox{.84}{$\square$}}}}

\newcommand{\TT}{\mathcal{T}(\mathbf{q})}
\usepackage{cite}
\usepackage{etoolbox}
\AtBeginEnvironment{quote}{\singlespace\vspace{-\topsep}\small}
\AtEndEnvironment{quote}{\vspace{-\topsep}\endsinglespace}
\usepackage{setspace}
\usepackage{xcolor}

\title{Poset Hopf Monoids}

\author{Mario Sanchez}
\date{May 27, 2020}
\thanks{University of California, Berkeley: mario\_sanchez@berkeley.edu. \newline The author was supported by the NSF Graduate Fellowship DGE 1752814.}

\begin{document}
\maketitle
\begin{abstract}{

We initiate the study of a large class of species monoids and comonoids which come equipped with a poset structure that is compatible with the multiplication and comultiplication maps. We show that if a monoid and a comonoid are related through a Galois connection, then they are dual to each other. This duality is best understood by introducing a new basis constructed through M\"obius inversion. We use this new basis to give uniform proofs for cofreeness and calculations of primitives for the Hopf monoids of set partitions, graphs, hypergraphs, and simplicial complexes.

Further, we show that the monoid and comonoid of a Hopf monoid are related through a Galois connection if and only if the Hopf monoid is linearized, commutative, and cocommutative. In these cases, we give a grouping-free formula for the antipode in terms of an evaluation of the characteristic polynomial of a related poset. This gives new proofs for the antipodes of the Hopf monoids of graphs, hypergraphs, set partitions, and simplicial complexes.
}
\end{abstract}
\section{Introduction}

In the last few years there has been much work on the combinatorics of Hopf algebras. These algebraic objects abstract the notion of a family of objects with a "merging" operation and a "breaking" operation.  Many families of combinatorial objects have been realized as Hopf algebras including graphs, symmetric function families, generalized permutahedra \cite{GPHopf}, simplicial complexes \cite{Simp}, pipe dreams \cite{HopfDreams}, representations of towers of groups \cite{zelevinsky2006representations}, etc. 

 There are many examples of Hopf algebras where there is a natural poset structure on a basis which respects the algebraic structure. In these cases, it is useful to construct a new basis using M\"obius inversion. A classical example of this is quasisymmetric function theory where we have two bases, the monomial basis $M_{\alpha}$ and the fundamental basis $F_{\alpha}$ which are related to each other by M\"obius inversion. In a different context, Crapo and Schmitt \cite{PrimMatroids} define a new basis for the Hopf algebra of matroids using M\"obius inversion in the weak order poset of matroids. Using this new basis, the authors give a new description for the space of primitives and a new proof of cofreeness.

These commonalities led to the following comment by Crapo and Schimitt \cite{PrimMatroids}.

\begin{quote}
    Many of the Hopf algebras now of central importance in algebraic combinatorics share certain striking features, suggesting the existence of a natural, yet-to-be-identified, class of combinatorial Hopf algebras. These Hopf algebras are graded and cofree, each has a canonical basis consisting of, or indexed by, a family of (equivalence classes of) combinatorial objects that is equipped with a natural partial ordering, and in each case the algebraic structure is most clearly understood through the introduction of a second basis, related to the canonical one by Mobius inversion over the partial ordering.
\end{quote}

This suggests two questions. First, what is this class of combinatorial Hopf algebras? Namely, what compatibility condition between the poset on the basis and the Hopf algebra is present in all of these examples? Second, why is defining a new basis using M\"obius inversion useful in the study of the algebraic structure? In this paper, we answer both of these questions as well as giving new examples of this phenomenon. 

To resolve the first question, it is not sufficient to restrict ourselves to the category of Hopf algebras. It is difficult to relate the poset structure of the basis to the multiplication and comultiplication maps. This is because these maps typically output linear combinations of basis elements, which forces us to extend the poset of the basis to a poset of linear combinations; however, it is not clear how to do this. To bypass this problem, we extend Joyal's theory of species to the category of posets and initiate the study of poset monoids, poset comonoids, and poset Hopf monoids. 

Define a \textbf{poset species} $\mbf{F}$ to be an assignment of a poset $\mbf F[I]$ to each finite set $I$ subject to some relabelling axioms. Then a \textbf{poset monoid} is a species $\mbf{F}$ alongside an order-preserving multiplication map $m_{S,T}: \mbf{F}[S] \times \mbf{F}[T] \to \mbf{F}[S \sqcup T]$ and a \textbf{poset comonoid} is a species $\mbf{F}$ alongside an order-preserving comultiplication map $\Delta_{S,T}: \mbf{F}[S \sqcup T] \to \mbf{F}[S] \times \mbf{F}[T]$ both satisfying some additional properties. Framing the examples that Crapo and Schmitt studied in this context, reveals that in all of these examples, there is a poset monoid $(\mbf{F}, \Box)$ and a poset comonoid $(\mbf{F},\Delta)$ on the same poset species where $m_{S,T}$ and $\Delta_{S,T}$ form a Galois connection. That is,
 \[ \Delta_{S,T}(x) \leq (y,z) \iff x \leq y \Osq_{S,T}z.\]
We call this pair $(\mbf F,\Delta)$ and $(\mbf F, \Box)$ an \textbf{adjoint pair}.

We can linearize a poset species $\mbf{F}$ to obtain a vector space $\kk \mbf{F}[I] = \text{span}(\{x \in \mbf{F}[I]\})$. This is equipped with the linearized basis $\{x \in \mbf{F}[I]\}$. Define the \textbf{inverted basis} of $\kk \mbf{F}[I]$ as the collection of elements
 \[ \omega_{x} = \sum_{x \leq y} \mu(x,y) \; y\]
for $x,y \in \mbf{F}[I]$, where $\mu(-,-)$ is the M\"obius function of the poset $\mbf{F}[I]$. Then, the answer to the second question is given by the following theorem.
\begin{theorem} Let $(\mbf{F}, \Delta)$ be a poset comonoid and $(\mbf{F}, \Box)$ be a poset monoid that form an adjoint pair, then
 \[ (\kk \mbf{F}, \Box) \cong (\kk \mbf{F}^*, \Delta^*),\]
with explicit isomorphism $x \mapsto \omega_x^*$.
\end{theorem}
In other words, if a poset monoid and a poset comonoid form an adjoint pair, then their linearizations are dual. Further, the linearized basis of the monoid and the inverted basis of the comonoid are dual bases. This means that any property of the linearized basis of the monoid can be dualized into the dual property of the inverted basis of the comonoid. Applying this to the indecomposables of the monoid gives the following description of the primitives of the comonoid.

\begin{theorem}Let $(\mbf{F},\Delta)$ and $(\mbf{F},\Box)$ be a monoid and a comonoid that form an adjoint pair. An inverted basis element $\omega_x \in \kk \mbf{F}[I]$ is primitive if and only if $x$ is indecomposable in the monoid $(\mbf{F}, \Box)$. Further, the set $\{\omega_x \lvert \; \text{$x \in \mbf{H}[I]$ and $x$ is $\Box$-indecomposable}\}$ is a basis for the space of primitives of $\kk\mbf{F}[I]$.
 
\end{theorem}
\noindent
Applying it to freeness and cofreeness gives the following.

\begin{theorem} Let $(\mbf{C},\Delta)$ be a poset comonoid and $(\mbf{C},\Box)$ be a commutative poset monoid that form an adjoint pair. If every element in  $\mbf{C}$ has a unique $\Box$-factorization, then $(\kk\mbf{C}, \Delta)$ is cofree.
\end{theorem}
With these theorems, we give new descriptions of the primitives of graphs and set partitions as well as the first calculations of the primitives of hypergraphs and simplcial complexes. We also give new proofs of the cofreeness of these examples.

We obtain stronger results by focusing on the situation where the adjoint pair $(\mbf{F}, \Delta)$ and $(\mbf{F}, \Box)$ also form a Hopf algebra $(\mbf{F}, \Box, \Delta)$. We call these \textbf{self-adjoint Hopf monoids}. These poset Hopf monoids admit a, possibly different, poset structure $\leq_r$ called the \textbf{reassembly poset}, first studied by Marberg \cite{Marberg}, given by $x \leq y$ if and only if $y = m_{S_1,\cdots,S_k} \circ \Delta_{S_1,\cdots,S_k}(x)$ for some ordered set partition $S_1 \sqcup \cdots \sqcup S_k = I$. We give the following grouping-free formula for their antipodes.

\begin{theorem} Let $\mbf{H}$ be a self-adjoint Hopf monoid and $\leq_r$ be the reassembly poset. Let $[x,y]$ denote the interval from $x$ to $y$ in the reassembly poset. Let $p_{[x,y]}(t)$ be the characteristic polynomial of this interval. Then, the antipode for $\mbf{H}$ is given by
 \[ S_I(x) = \sum_{x \leq_r y} p_{[x,y]}(-1) \; y\]
\end{theorem}
Applying this to graphs, hypergraphs, simplicial complexes, and set partitions gives a new description of existing results.

Finally, we give the following classification of self-adjoint Hopf monoids.
\begin{theorem} A set Hopf monoid is self-adjoint with respect to some poset structure if and only if it is commutative and cocommutative.
\end{theorem}

The structure of the paper is as follows. In Section 2, we extend Joyal's theory of species to the category of posets as well as defining poset Hopf monoids. In Section 3, we study general results of adjoint pairs $(\mbf{F},\Delta)$ and $(\mbf{F},\Box)$, such as finding a basis for the primitives of the comonoid in terms of the monoid and giving a criteria for when the comonoid is cofree in terms of $\Box$. We also show how to use the Fock functor to translate our results to the category of combinatorial Hopf algebras. In Section 4, we study self-adjoint Hopf monoids. We study properties of the reassembly poset and use it to give a classification as the Hopf monoids which are linearized, commutative, and cocommutative. We finish the section by giving the formula for the antipode in terms of a characteristic polynomial. In Section 5, we apply our theory to the examples of set partitions, simplicial complexes, and hypergraphs.
\section{Poset Hopf Monoids}\label{sec:Poset Hopf Monoids}

One of the major obstacles in trying to combine Hopf algebras with posets is that while many Hopf algebras have a basis with a natural poset structure, it is not clear what it means for the multiplication and comultiplication maps to be order-preserving. This is because these two operations generally output linear combinations of basis elements. Instead of trying to extend the poset structure to these linear combinations, we take a different route and define the notion of poset Hopf monoids, which allows us to bypass this problem.


\subsection{Poset Background} 
We focus on posets that are locally finite, that is, given $x,z$ in the poset $P$, the interval $[x,z] = \{x \leq y \leq z \; \lvert \; y \in P\}$ is finite. The \textbf{M\"obius function} $\mu$ of a poset $P$ is a map from intervals of $P$ to $\Z$ defined recursively by 
\[ \mu_P(x,y) = \left \{ \begin{array}{ll}
                1 & \text{if $x = y$} \\
                -\sum_{x \leq z < y}\mu_P(x,z) & \text{otherwise}.
                \end{array} \right.\]
In this context, we have Rota's M\"obius Inversion Theorem \cite{Rota1} which tells us that for two integer-valued functions, $f,g:P \to \Z$ we have 
  \[ f(x) = \sum_{x \leq y} g(y) \; \; \text{for all $x$} \iff g(x) = \sum_{x \leq y} \mu(x,y) f(y) \; \;\text{for all $x$.}\]
Given two posets $P,Q$, a \textbf{Galois connection} $f \dashv g$ is a pair of order-preserving maps $(f,g)$ with $f:P \to Q$ and $g: Q \to P$ with the property that for all $x \in P$ and all $y \in Q$,
 \[ f(x) \leq_Q y \iff x \leq_P g(y).\]
We say $f$ is the left-adjoint to $g$ and $g$ is the right-adjoint to $f$. An order-preserving map has at most one left-adjoint or right-adjoint.

The main utility of Galois connections comes from the following lemma due to Rota.
\begin{lemma}\cite{Rota1} If $P,Q$ are posets with $x \in P$ and $b \in Q$ and $f \dashv g$ is a Galois connection then,
\[ \sum_{\substack{x \leq y \\ f(y) = b}} \mu_P(x,y) = \sum_{\substack{a \leq b \\ g(a) = x}} \mu_Q(a,b).\]
\end{lemma}

Let $\kk$ be a field of characteristic $0$ and $P$ be a poset. We call the free vector space $\kk P$ the \textbf{linearization} of $P$. This is a vector space equipped with a basis indexed by the elements of $P$. We refer to this basis as the \textbf{linearized basis}. We will abuse notation by using the same letter to denote the elements of $P$ and the basis vectors of $\kk P$. Define the \textbf{inverted basis} by
 \[ \omega_x = \omega_x^P = \sum_{x \leq y} \mu(x,y) y.\]
By M\"obius inversion, the set $\{\omega_x\}_{x \in P}$ is a basis of $\kk P$. The following lemma shows that if $f \dashv g$ is a Galois connection between two posets $P$ and $Q$, then the image of $\omega_x^P$ under $f$ has a natural description.

\begin{lemma}\label{lemma:Gal} Let $P$ and $Q$ be posets with a Galois connection $f \dashv g$. Let $\omega_x^P$ be the inverted basis for $x \in P$ and $\omega_y^Q$ be the inverted basis for $y \in Q$. Then,
 \[ f(\omega_x^P) = \sum_{a \; \lvert \; g(a) = x} \omega_a^Q\]
\end{lemma}

\begin{proof} This follows from direct calculation with Rota's lemma. We have

\begin{align*}
f(\omega_x^P) &= \sum_{x \leq y} \mu(x,y) f(y) \\
    &= \sum_{b \in Q} \left (\sum_{\substack{x \leq y \\ f(y) = b }} \mu(x,y)\right ) b \\
    &= \sum_{b \in Q} \left ( \sum_{\substack{a \leq b \\ g(a) = x}} \mu(a,b)\right )b
\end{align*}
where the last equality uses Rota's lemma. Switching the order of the sums gives,
\begin{align*}
    f(\omega_x^P) &= \sum_{a \; \lvert \; g(a) = x} \; \;\sum_{a \leq b}\mu(a,b) \; b = \sum_{a \; \lvert \; g(a) = x} \omega_a^Q.
\end{align*}
\end{proof}

\subsection{Poset Species}

We now extend Joyal's theory of species \cite{Joyal} to posets.

\begin{definition} A poset species $\mathbf{F}$ is a functor from the category of finite sets with bijections to the category of posets that is a singleton on the empty set. This consists of the following data.

\begin{itemize}
    \item For each finite set $I$, a poset $\mathbf{F}[I]$ called the structures of type $F$ on label set $I$.
    \item For each bijection $f:I \to J$ an order-preserving bijection $\mbf{F}[f]$
     \[\mbf{F}[I] \to \mbf{F}[J], \]
     such for two bijections $f:I \to J$ and $g:J \to K$, we have
      \[ \mbf{F}[g \circ f] = \mbf{F}[g] \circ \mbf{F}[f],\]
     and $\mbf{F}[\text{id}] = \text{id}$.
\end{itemize}
We say that $\mathbf{F}$ is connected if $|\mbf{F}[\emptyset]| = 1$. Throughout this paper, every poset species will be assumed to be connected.
\end{definition}

It is important to notice that for any species $F$ we have an order-preserving action of Aut$(I)$ on $\mbf{F}[I]$. This is called the \textbf{relabelling action} on $\mbf{F}[I]$.

\begin{example} We will follow the example of graphs throughout this paper. For a finite set $I$, let $I \choose 2$ denote the set of distinct pairs of elements in $I$. Define a poset species by
 \[ \mbf{Graphs}[I] = \left \{ G \; | \; \text{ $G = \left ( I, E(G) \subset {I \choose 2} \right )$ is a simple graph on index set $I$} \right \}.\]
 
The poset structure on graphs is given by $H \leq G$ if and only if $E(H) \subset E(G)$. Any bijection $f:I \to J$ induces a bijection between $I \choose 2$ and $J \choose 2$, hence a bijection between $\mbf{G}[I]$ and $\mbf{G}[J]$ that simply relabels the vertices. It is clear that this relabelling action is order-preserving.
\end{example}

We also need the definition of a morphism of poset species.

\begin{definition} Let $\mbf{F_1}$ and $\mbf{F_2}$ be two poset species. A \textbf{poset species morphism} $\alpha$  from $\mbf{F_1}$ to $\mbf{F_2}$ is a natural transformation of functors. In other words, $\alpha$ is a collection of poset maps $\alpha[I]: \mbf{F_1}[I] \to \mbf{F_2}[I]$ such that the following diagram commutes

\begin{center}
\begin{tikzcd}
\mbf{F_1}[I] \arrow{r}{\alpha[I]} \arrow{d}{\mbf{F_1}[f]}
&\mbf{F_2}[I] \arrow{d}{\mbf{F_2}[f]}\\
\mbf{F_1}[J] \arrow{r}{\alpha[J]} & \mbf{F_2}[J]
\end{tikzcd}
\end{center}
for any two sets $I,J$ and any bijection $f: I \to J$.
\end{definition}
A poset species isomorphism is a poset species morphism $\phi$ where each $\phi_I$ is an order-preserving bijection.

We will still need the more usual variants of species.

\begin{definition} A (connected) \textbf{vector space species} is a functor from the category of finite sets with bijections to vector spaces. A (connected) \textbf{set species} is a functor from the category of finite sets with bijections to finite sets.
\end{definition}
A \textbf{bilinear form} on a vector space species denoted by $\langle -,- \rangle$ is a bilinear $\langle - ,- \rangle_I$ which respects the relabeling action. The \textbf{direct sum} of two vector space species $\mbf{F}$ and $\mbf{G}$ is given by $\mbf{F} \oplus \mbf{G}[I] = \mbf{F}[I] \oplus \mbf{G}[I].$

More importantly, if $\mbf{F}$ is a poset species, its \textbf{linearization} $\kk \mbf{F}$ is a vector space species defined by $(\kk \mbf{F})[I] = \kk \mbf{F}[I]$. A \textbf{basis} for a vector species $\mbf{F}$ is a set species $\mbf{b}$ such that $\mbf{b}[I]$ is a basis of $\mbf{F}[I]$. The invertible basis of a linearized poset species $\kk \mbf{F}$ is given by $\mbf{w}[I] = \{\omega_x\}_{x \in \mbf{F[I]}}$.


\subsection{Monoids, Comonoids, and Hopf Monoids}
We can now define monoids and comonoids in the category of poset species. The main difference between this and the usual theory is that all of the structure maps are order-preserving.

\begin{definition} A (connected) \textbf{poset monoid} $(\mbf{M},m)$ is a poset species equipped with a collection of order-preserving maps
 \[ m_{S, T}: \mbf{M}[S] \times \mbf{M}[T] \to \mbf{M}[I],\]
for each ordered set partition of $I$, where the poset structure on $\mbf{M}[S] \times \mbf{M}[T]$ is the product poset. This operations must satisfy the following axioms:
\vskip 2ex
\noindent \textit{(Naturality)} Let $I$ and $J$ be two sets and $f: I \to J$ be a bijection. Let $I = S \sqcup T$ be a decomposition and let $f|_S$ and $f|_T$ be the restrictions of $f$ to $S$ and $T$, respectively. This gives us a decomposition of $J = f(S) \sqcup f(T)$ and a pair of bijections $f|_S: S \to f(S)$ and $f|_T T \to f(T)$ Then, we have the following commutative diagram
\begin{center}
\begin{tikzcd}[column sep=large]
\mbf{M}[S] \otimes \mbf{M}[T] \arrow{r}{m_{S,T}} \arrow{d}{\mbf{M}[f|_S] \otimes \mbf{M}[f|_T]}
&\mbf{M}[I] \arrow{d}{\mbf{M}[f]}
\\
\mbf{M}[f(S)] \otimes \mbf{M}[f(T)] \arrow{r}{m_{f(S), f(T)}} 
&\mbf{M}[J]
\end{tikzcd}
\end{center}
\noindent \textit{(Unitality)} $\mbf{M}[\emptyset]$ has a single element. We will denote that element by $1$. For any $x \in \mbf{M}[I]$ and for the two trivial decompositions $I = I \sqcup \emptyset$ and $I = \emptyset \sqcup I$, we have
\begin{align*}
1 \cdot x = x \cdot 1 &= x
\end{align*}
 \noindent \textit{(Associativity)} Let $I = S \sqcup T \sqcup R$ be a decomposition of the index set $I$. Then the following diagram commutes

\begin{center}
\begin{tikzcd}[column sep=large]
\mbf{M}[S] \otimes \mbf{M}[T] \otimes \mbf{M}[R] \arrow{r}{\text{id} \;\otimes\; m_{T,R}} \arrow{d}{m_{S,T} \;\otimes \;\text{id}}
& \mbf{M}[S] \otimes \mbf{M}[T \sqcup R] \arrow{d}{m_{S, T \sqcup R}} 
\\
\mbf{M}[S \sqcup T] \otimes \mbf{M}[R] \arrow{r}{m_{S,T \sqcup R}}
&\mbf{M}[I]
\end{tikzcd}
\end{center}
This allows us to define a multiplication map $m_{S_1,S_2, \cdots, S_k}$ for any ordered set partition of $I$.

\end{definition}

\begin{example} Given two graphs $G$ and $H$ on different label sets $S$ and $T$, we define their disjoint union by $G \sqcup H = (S \sqcup T, E(G) \sqcup E(H))$. This is order-preserving. This makes $\mbf{Graphs}$ into a poset Hopf monoid with $m_{S,T}(G,H) = G \sqcup H$. The unit is the empty graph $1 = ( \emptyset, \emptyset)$.
\end{example}

A monoid homomorphism between two monoids $\mbf{M_1}$ and $\mbf{M_2}$ is a poset species homomorphism $\phi$ such that
 \[ m_{S,T}(\phi(x),\phi(y)) = \phi(m_{S,T}(x,y)).\]
A monoid isomorphism is a monoid homomorphism which is a poset species isomorphism.

We also have the dual notion of a monoid.

\begin{definition} A (connected) \textbf{poset comonoid}\footnote{This name is somewhat misleading since a poset comonoid is not a comonoid in the category of poset species. However, the linearization of a poset comonoid has a natural comonoid structure in the category of vector space comonoids. Further, every vector space comonoid which is linearized arises in this way from a poset comonoid.} $(\mbf{C},\Delta)$ is a poset species equipped with a collection of order-preserving maps
 \[ \Delta_{S,T}: \mbf{C}[I] \to \mbf{C}[S] \times \mbf{C}[T]. \]
for each ordered set partition of $I$, where the poset structure on $\mbf{C}[S] \times \mbf{C}[T]$ is the product poset. This operations must satisfy the following axioms:
\vskip 2ex
\noindent
\textit{(Naturality)} Let $I$ and $J$ be two sets and $\sigma: I \to J$ be a bijection. Let $I = S \sqcup T$ be a decomposition and let $\sigma|_S$ and $\sigma|_T$ be the restrictions of $\sigma$ to $S$ and $T$, respectively. Then, we have the following commutative diagrams

\begin{center}
\begin{tikzcd}[column sep=large]
\mbf{C}[I] \arrow{r}{\Delta_{S,T}} \arrow{d}{\mbf{C}[\sigma]}
&\mbf{C}[S] \otimes \mbf{C}[T] \arrow{d}{\mbf{C}[\sigma|_T] \otimes \mbf{C}[\sigma|_S]}\\
\mbf{C}[J] \arrow{r}{\Delta_{\sigma(S), \sigma(T)}}
&\mbf{C}[\sigma(S)] \otimes \mbf{C}[\sigma(T)]
\end{tikzcd}
\end{center}
\noindent
\textit{(Counitality)} $\mbf{C}[\emptyset]$ has a single element. We will denote that element by $1$. For any $x \in \mbf{C}[I]$ and for the two trivial decompositions $I = I \sqcup \emptyset$ and $I = \emptyset \sqcup I$, we have
\begin{align*}
\Delta_{I,\emptyset}(x) &= x \otimes 1,\\
\Delta_{\emptyset,I}(x) &= 1 \otimes x.
\end{align*}
\noindent
\textit{(Coassociativity)} Let $I = S \sqcup T \sqcup R$ be a decomposition of the index set $I$ into three. Then the following diagram commutes

\begin{center}
\begin{tikzcd}[column sep=large]
\mbf{C}[I] \arrow{r}{\Delta_{S, T \sqcup R}} \arrow{d}{\Delta_{S \sqcup T, R}}
&\mbf{C}[S] \otimes \mbf{C}[T \sqcup R] \arrow{d}{\text{id} \; \otimes \; \Delta_{T,R}}
\\
\mbf{C}[S \sqcup T] \otimes \mbf{C}[R] \arrow{r}{\Delta_{S,T} \; \otimes \; \text{id}}
&\mbf{C}[S] \otimes \mbf{C}[T] \otimes \mbf{C}[R]
\end{tikzcd}
\end{center}

\end{definition}

A comonoid homomorphism between two comonoids $\mbf{C_1}$ and $\mbf{C_2}$ is a poset species map $\phi$ such that
\[\Delta_{S,T}(\phi_I(x)) = (\phi_S, \phi_T) \circ (\Delta_{S,T}(x)).\]
An isomorphism is an comonoid homomorphism that is a poset species isomorphism.

\begin{example} Let $G$ be a graph on label set $I$ and $S \subset I$. Then, the restriction of $G$ to $S$ is given by $G|S = \left ( S, E(G) \cap {S \choose 2} \right )$. This gives $\mbf{Graph}$ the structure of a poset comonoid by
 \[\Delta_{S,T}(G) = \left ( G|S, G|T \right ). \]
Since restriction is order-preserving, graphs form a poset comonoid.
\end{example}

\begin{definition} A \textbf{poset Hopf monoid} is a (connected) poset species that is a monoid and a comonoid and has the additional axiom:
\vskip 2ex

\noindent \textit{(Compatibility)} Let $I = S_1 \sqcup S_2$ and $I = T_1 \sqcup T_2$ be two decompositions of $I$. Let $A = S_1 \cap T_1$, $B= S_1 \cap T_2$, $C = S_2 \cap T_1$, and $D = S_2 \cap T_2$.Then, we have the following commutative diagram

\begin{center}
\begin{tikzcd}[column sep=tiny]
\tb{H}[S_1] \otimes \tb{H}[S_2] \arrow{r}{m_{S_1, S_2}} \arrow{d}{\Delta_{A, B}\; \otimes\; \Delta_{C, D}}
& \tb{H}[I] \arrow{r}{\Delta_{T_1,T_2}}
& \tb{H}[T_1] \otimes \tb{H}[T_2] \\
\tb{H}[A] \otimes \tb{H}[B] \otimes \tb{H}[C] \otimes \tb{H}[D] \arrow{rr}{\text{id} \; \otimes \; \beta \; \otimes \ \text{id}}
&
&\tb{H}[A] \otimes \tb{H}[C] \otimes \tb{H}[B] \otimes \tb{H}[D] \arrow{u}{m_{A,C}\; \otimes \;m_{B,D}}
\end{tikzcd}
\end{center}
where $\beta$ is the braiding map $\beta(x,y) = (y,x)$.
\end{definition}

A monoid is \textbf{commutative} if $m_{S,T}(x,y) = m_{T,S}(x,y)$. Dually, a comonoid is \textbf{cocommutative} if $\Delta_{S,T}(x) = \beta_{T,S} \Delta_{T,S}(x)$, where $\beta_{T,S}: \mbf{C}[T] \times \mbf{C}[S] \to \mbf{C}[S] \times \mbf{C}[T]$ is the map that switches factors.

We will still need the usual definitions of \textbf{vector space} monoids, comonoids, and Hopf monoids as well as \textbf{set} monoids, comonoids, and Hopf monoids. These are obtained by all the axioms above where $\mbf{F}$ is a set, vector space species and the maps are set, linear maps, respectively.

\begin{definition}
The \textbf{antipode} of a vector space Hopf monoid $\mbf{H}$
is the map $S[I]: \mbf{H}[I] \to \mbf{H}[I]$ given by
 \[S[I](x) = \sum_{S_1 \sqcup \cdots \sqcup S_k} (-1)^k m_{S_1,\cdots,S_k} \circ \Delta_{S_1,\cdots,S_k}(x). \]

\end{definition}
In general this formula has a large amount of cancellation. Finding a simpler formula of the antipode is one of the major points of study in the theory of Hopf monoids. One reason for the interest in the antipode comes from its connection with combinatorial reciprocity of polynomial invariants \cite{GPHopf}.

\section{Adjoint Pairs} 
We present our main definition of adjoint pairs that describes a particular relationship between a poset monoid and a poset comonoid. We use this to understand duality, primitives, cofreeness. We then describe how our results descend to the category of combinatorial Hopf algebras.

\subsection{Main Definition and Result}

\begin{definition} A poset comonoid $(\mbf{F},\Delta)$ and a poset monoid $(\mbf{F},\Box)$ on the same poset species $\mbf{F}$ form an \textbf{adjoint pair} if $\Delta_{S,T}$ and $\Box_{S,T}$ are a Galois connection for all finite sets $S$ and $T$. This can happen in two ways. The first, denoted by $\Delta \dashv \Box$, is if for all $x \in \mbf{F}[I]$, $y \in \mbf{F}[S]$, and $ z \in \mbf{F}[T]$ we have
\[\Delta_{S,T}(x) \leq (y,z) \iff x \leq \Box_{S,T}(y,z),\]

\noindent The second, denoted by $\Box \dashv \Delta$, is if for all $x \in \mbf{F}[I]$, $y \in \mbf{F}[S]$, and $ z \in \mbf{F}[T]$ 
\[\Box_{S,T}(x,y) \leq z \iff (x,y) \leq \Delta_{S,T}(z) \]
\end{definition}
Clearly, if we reverse the order of the poset on $\mbf{F}$, then these two cases are the same. However, often the combinatorics forces a particular poset on us and reversing the order of the poset can be unnatural. 

\begin{example} Let $G_1$ be a graph on vertex set $S$ and $G_2$ be a graph on vertex set $T$. Then, define the \textbf{free product} of these two graphs $G_1 \Box G_2 = \Box_{S,T}(G_1,G_2)$ by
 \[ G_1 \Box G_2 = (G_1^c \sqcup G_2^c)^c,\]
where $G^c$ denotes the complement of the graph $G$. In other words, $G_1 \Osq G_2$ is the disjoint union of $G_1$ and $G_2$ with all the edges between $S$ and $T$ added.

It is quick to verify that
 \[ (G|S, G|T) \leq (G_1, G_2) \iff G \leq G_1 \Box G_2.\]
Thus, the comonoid of graphs $(\mbf{Graphs}, \Delta)$ and the monoid of graphs $(\mbf{Graphs},\Box)$ form an adjoint pair with $\Delta \dashv \Box$

This comultiplication $\Delta$ is part of a second adjoint pair. We have
 \[ G_1 \sqcup G_2 \leq G \iff (G_1, G_2) \leq \Delta_{S,T}(G) .\]
This implies that the comonoid $(\mbf{Graphs},\Delta)$ and the monoid $(\mbf{Graphs},m)$ form an adjoint pair with $m \dashv \Delta$.
\end{example}

The following simple result will be important throughout this paper.

\begin{theorem}\label{thm: inverted basis}
Let $(\mbf{F},\Delta)$ be a poset comonoid and $(\mbf{F}, \Box)$ be a poset monoid which form an adjoint pair with $\Delta \dashv \Box$. Let $\omega_x = \sum_{x \leq y} \mu(x,y) \; y$ be an inverted basis element of $\kk \mbf{F}$. Then,
    \[\Delta_{S,T}(\omega_x) = \sum_{\substack{x_1 \Box_{S,T} x_2 = x}} \omega_{x_1} \otimes \omega_{x_2}. \]
\end{theorem}

\begin{proof}

We use the Galois connection $\Delta_{S,T} \dashv \Box_{S,T}$ and Lemma \ref{lemma:Gal} to study the comonoid structure. 

By a direct application, we have
 \[ \Delta_{S, T}(\omega_x) = \sum_{x_1 \Osq x_2 = x} \omega_{(x_1, x_2)},\]
where $\omega_{(x_1,x_2)}$ is the inverted basis in the linearization of the product poset $\kk (\mbf{C}[S] \times \mbf{C}[T]) = \kk\mbf{C}[S] \otimes \kk\mbf{C}[T]$. Notice,
\begin{align*}
\omega_{(x_1,x_2)}  &=
    \sum_{(x_1,x_2)\leq (y_1 ,y_2)} \mu((x_1, x_2), (y_1 , y_2))\; y_1 \otimes y_2 \\
    &= \sum_{(x_1,x_2) \leq (y_1,y_2)} \mu(x_1,y_1) \mu(x_2,y_2) y_1 \otimes y_2 \\
    &= \left ( \sum_{x_1 \leq y_1} \mu(x_1,y_1) y_1\right ) \otimes \left ( \sum_{x_2 \leq y_2} \mu(x_2,y_2) y_2\right ) \\
    &= \omega_{x_1} \otimes \omega_{x_2}
\end{align*}

This gives us
 \[ \Delta_{S, T}(\omega_x) = \sum_{x_1 \Osq x_2} \omega_{x_1} \otimes \omega_{x_2}.\]
\end{proof}
\noindent If $\Box \dashv \Delta$, then the previous theorem still holds replacing $\omega_x$ by $\omega^x = \sum_{y \leq x} \mu(y,x) y$.

Notice that in the previous example the monoid $(\mbf{Graphs}, m)$ and the comonoid $(\mbf{Graphs},\Delta)$ form a Hopf monoid, that is they satisfy the compatibility axiom, whereas the monoid $(\mbf{Graphs}, \Box)$ and the comonoid $(\mbf{Graphs},\Delta)$ do not form a Hopf monoid. This leads to the following definition.

\begin{definition} A poset Hopf monoid $(\mbf{H},m,\Delta)$ is \textbf{self-adjoint} if the monoid $(\mbf{H},m)$ and the comonoid $(\mbf{H},\Delta)$ form an adjoint pair.
\end{definition}
We study these objects in detail in Section 4.

\subsection{Duality}
Given a vector space species $\mbf{F}$, its \textbf{linear dual} $\mbf{F}^*$ is given by $\mbf{F}^*[I] = \mbf{F}[I]^*$. Since duality of vector spaces is a contravariant functor, the linear dual of a monoid has a natural comonoid structure and vice-versa. One way of describing the dual multiplication and comultiplications is by choosing an explicit basis. If $\mbf{b}$ is a basis of $\mbf{F}$, then we can define a bilinear form $\langle -, - \rangle_I$ on $\mbf{F}[I]$ such that $\mbf{b}[I]$ is an orthogonal basis. Using this, we can build a dual basis $\mbf{b}^*$ for $\mbf{F}^*$. In this basis, the dual comultiplication $m_{S,T}^*: \mbf{F}[I] \to \mbf{F}[S] \otimes \mbf{F}[T]$ is given by
 \[ m_{S,T}^*(x^*) = \sum_{\substack{y \in \mbf{b}[S] \\ z \in \mbf{b}[T]}} \langle m_{S,T}(y,z), x\rangle \quad y^* \otimes z^*.\] 
 
Similarly, if we have a vector space comonoid $\mbf{F}$, then on the dual basis, the dual multiplication $\Delta_{S,T}^*: \mbf{F}[S] \otimes \mbf{F}[T] \to \mbf{F}[I]$ takes the form
 \[ \Delta_{S,T}^*(x^*,y^*) = \sum_{\langle x \otimes y, \Delta_{S,T}(z) \rangle \not = 0}\langle x \otimes y, \Delta_{S,T}(z) \rangle \quad z^*,\]
where we extend the bilinear form to the tensor product by $\langle x \otimes x', y \otimes y' \rangle = \langle x, y \rangle \cdot \langle x', y' \rangle$.
 
We say that a vector space monoid $(\mbf{F},m)$ and a vector space comonoid $(\mbf{G},\Delta)$ are \textbf{dual} if there is an isomorphism of monoids $\mbf{F} \cong \mbf{G}^*$. This isomorphism can be realized by a non-degenerate bilinear form $\langle -, -\rangle$ such that
    \[ \langle m_{S,T}(x,y), z \rangle = \langle x \otimes y, \Delta(z)\rangle.\]

A vector space Hopf monoid $(\mbf{H},m, \Delta)$ is \textbf{self-dual} if there exists an isomorphism of species $\mbf{H} \cong \mbf{H}^*$ that is simultaneously a monoid and comonoid isomorphism. This isomorphism can be realized as a symmetric non-degenerate bilinear form $\langle -, - \rangle_I$ such that
 \[\langle \Delta_{S,T}(x), y \otimes z \rangle = \langle x, m_{S,T}(y,z) \rangle. \]

\begin{definition} If $\mbf{F}$ is a poset species, then we can define a bilinear form $\langle -,- \rangle$ on $\kk \mbf{F}$ by
 \[ \langle x, y \rangle = \zeta(x,y)  = \begin{cases}
    1 & \text{ if $x \leq y$} \\
    0 & \text{else}
    \end{cases}\]
for $x,y \in \mbf{F}[I]$. We call this the \textbf{zeta bilinear form} of $\mbf{F}$. 

\end{definition}

Then, if $(\mbf{F}, \Delta)$ is a poset comonoid and $(\mbf{F}, \Box)$ is a poset monoid that form an adjoint pair with $\Delta \dashv \Box$, by the definition of Galois connection we have
\[ \langle \Delta_{S,T}(x), y \otimes z \rangle = \langle x, y \Box_{S,T} z \rangle.\]

Combining this with Theorem \ref{thm: inverted basis}, we obtain the following result.

\begin{theorem}\label{thm:duality} Let $(\mbf{F}, \Delta)$ be a poset comonoid and $(\mbf{F},\Box)$ be a poset monoid that form an adjoint pair with $\Delta \dashv \Box$, then $(\mbf{F},\Box) \cong (\mbf{F}^*, \Delta^*)$ with explicit isomorphism $x \mapsto \omega_x^*$. 
\end{theorem}

One thing to notice is that for a self-adjoint poset Hopf monoid $(\mbf{H}, m, \Delta)$ this gives an isomorphism between the monoid $(\mbf{H},m)$ and the monoid $(\mbf{H}^*,\Delta)$. Similarily, by reversing the poset order, we have an isomorphism between $(\mbf{H}, \Delta)$ and $(\mbf{H}^*, m^*)$. This does not immediately imply that $\mbf{H}$ is self-dual because these two isomorphisms are not the same species map. However, in Section 4 we see that it is still true that every self-adjoint poset Hopf monoid is self-dual.

\subsection{Primitives}

We now turn to Crapo and Schmitt's observation that posets seem to help calculate primitives of Hopf algebras.

\begin{definition} Let $(\mbf{M},m)$ be a poset monoid. An element $x \in \mbf{M}[I]$ is \textbf{indecomposable} if it is not of the form $x = m_{S,T}(y,z)$ for some $I = S \sqcup T$ and $y \in \mbf{M}[S]$ and $z \in \mbf{M}[T]$. For $\kk \mbf{M}$ with linearized basis $\mbf{b}$, we can define a bilinear form that makes $\mbf{b}$ orthonormal. Then, the subspace of indecomposables of $\kk \mbf{M}[I]$ is given by the elements $x$ such that
 \[ \langle x, m_{S,T}(y,z) \rangle = 0,\]
for all $I = S \sqcup T$ and $y \in \kk \mbf{M}[S]$ and $z \in \kk \mbf{M}[T]$. This is the same as the vector space spanned by the basis elements $x \in \mbf{M}[I]$ which are indecomposable. 
\end{definition}

\begin{definition} Let $(\mbf{C}, \Delta)$ be a vector space comonoid. Then an element $x \in \mbf{C}$ is \textbf{primitive} if $\Delta_{S,T}(x) = 0$ for all non-trivial $I = S \sqcup T$. The subspecies of primitives of $\mbf{C}$ is denoted by $\mathcal{P}(\mbf{C})$.
\end{definition}

We now show that for an adjoint pair, the primitives of one determine the indecomposables of the other.

\begin{theorem}\label{thm:prim} Let $(\mbf{F},\Delta)$ and $(\mbf{F},\Box)$ be a monoid and a comonoid that form an adjoint pair with $\Delta \dashv \Box$. An inverted basis element $\omega_x \in \kk \mbf{F}[I]$ is primitive if and only if $x$ is indecomposable in the monoid $(\mbf{F}, \Box)$. Further, the set $\{\omega_x \lvert \; \text{$x \in \mbf{H}[I]$ and $x$ is $\Box$-indecomposable}\}$ is a basis for $\mathcal{P}(\kk\mbf{F})[I]$.
\end{theorem}

\begin{proof} Let $\langle -, -\rangle$ denote the zeta bilinear form on $\mbf{F}$. Then, we have
 \[ \langle \Delta(x), y \otimes z \rangle = \langle x, y \Box_{S,T} z \rangle\]
for any $x \in \kk \mbf{F}[I], y \in \kk \mbf{F}[S], z \in \kk \mbf{F}[T]$. For any inverted basis element $\omega_x$ and any linearized basis element $y$, we have
\begin{align*}
    \langle \omega_x, y \rangle &= \sum_{x \leq z} \mu(x,z) \zeta(z,y) \\
    &= \sum_{x \leq z \leq y} \mu(x,z) \\
    &= \begin{cases}
    1 & \text{if $y = x$} \\
    0 & \text{else} \\
    \end{cases}
\end{align*}

Combining these two facts, we have that for the following expression
 \[ \langle \Delta_{S,T}(\omega_x), y \otimes z \rangle = \langle \omega_x, y \Box_{S,T} z \rangle,\]
the right-hand-side is zero for all $y$ and $z$ if and only if $x$ is indecomposable and the left-hand-side is zero for all $y$ and $z$ if and only if $\omega_x$ is primitive. This implies the result.

\end{proof}

\begin{example} For $\mbf{Graphs}$, we have seen that $\Delta \dashv \Box$ where $\Box$ is the free product of graphs. This means that the subspecies of primitives of $\mbf{Graphs}$ has a basis given by $\omega_G$ for graphs that are indecomposable under the free product. With the realization that
 \[(G^{c} \Osq H^{c})^{c} = G \sqcup H, \]
we can identify these $\Box$-indecomposable graphs as the complements of connected graphs. Thus, a basis of the primitives is given by elements of the form
 \[ \omega_G = \sum_{G \leq H} (-1)^{E(H) - E(G)} H,\]
for $G$ whose complement is connected.
\end{example}

\subsection{Cofreeness and Freeness}

We now show why passing to the inverted basis helps proves cofreeness and freeness.

Let $\mbf{q}$ be a set species with the property that $\mbf{q}[\emptyset] = 0$. The \textbf{tensor species} on $\mbf{q}$ denoted by $\mathcal{T}(\mbf{q})$ is a vector space species given by

\[ \TT[I] = \kk \{(A \vDash I, x_1, x_2, \cdots, x_k) \; \lvert \; x_i \in \mbf{q}[A_i], \}\]
where $A \vDash I$ denotes that $A$ is an ordered set partition of $I$. 

The \textbf{free monoid} on $\mbf{q}$ is the species $\TT$ with the multiplication
\[ (A \vDash S, x_1, \cdots, x_k) \cdot (B \vDash T, y_1, \cdots, y_\ell) = (A \sqcup B \vDash I, x_1, \cdots, x_k, y_1, \cdots, y_\ell).\]
We say that $S \subset I$ is an initial segment of an ordered set partition $A \vDash I$ if there is some integer $i$ such that $(A_1 \sqcup A_2 \sqcup \cdots \sqcup A_i) = S$. In this case, we say $A|S = (A_1 \sqcup A_2 \sqcup \cdots \sqcup A_i)$ and $A/S = (A_{i+1}  \sqcup \cdots \sqcup A_k)$ Then, the \textbf{free comonoid} on $\mbf{q}$ is the species $\TT$ with the comultiplication map given by
 \[\Delta_{S,T}((A \vDash I, x_1, \cdots, x_k)) = \left \{ 
 \begin{array}{ll}
 (A|S, x_1, \cdots, x_i) \otimes (A/S, x_{i+1}, \cdots, x_k) & \text{if $S$ is an initial segment of $A$} \\
 0 & \text{else}
 \end{array} \right.
 \]

\begin{definition} A vector space monoid is free if for some $\mbf{q}$ it is isomorphic to $\TT$ as a monoid. A vector space comonoid is cofree if for some $\mbf{q}$ it is isomorphic to $\TT$ as a comonoid.
\end{definition}

A standard theorem is that a vector space comonoid is cofree if and only if its dual monoid is free. We say that a set monoid $(\mbf{M},\Box)$ has the unique factorizaiton property if every element $x \in \mbf{M}[I]$ has a unique expression as $x = \Box_{S_1,\cdots,S_k}(x_1,\cdots,x_k)$ for $x_i$ indecomposable. Clearly, if a set monoid $(\mbf{M},\Box)$ has the unique factorization property, it is free.

\begin{theorem} Let $(\mbf{C},\Delta)$ be a poset comonoid and $(\mbf{C},\Box)$ be a poset monoid that form an adjoint pair with $\Delta \dashv \Box$. If $(\mbf{C},\Box)$ has the unique factorization property, then $(\kk\mbf{C}, \Delta)$ is cofree.
\end{theorem}

\begin{proof} This follows from the observations above and from theorem \ref{thm:duality} which says that the dual of $(\kk \mbf{C},\Delta)$ is isomorphic to $(\kk \mbf{C},\Box)$.
\end{proof}

There is a similar notion for (co)commutative (co)monoids. Let $\mbf{q}$ be a set species with $\mbf{q}[\emptyset]$. Define a vector space species $\mathcal{S}(\mbf{q})$ by
 \[ \mathcal{S}(\mbf{q})[I] = \mbf{k} \left \{(A \vdash I, \{x_i\}) \; \lvert \; x_i \in \mbf{q}[A_i] \right \},\]
where $A \vdash I$ denotes an unordered set partition of $I$. The \textbf{free commutative monoid} on $\mbf{q}$ is the monoid structure on this species given by
 \[ (A \vdash S, \{x_i\}) \cdot (B \vdash T, \{y_i\}) = (A \sqcup B \vdash I, \{x_i\} \cup \{y_j\}).\]
Say that a subset $S \subset I$ is compatible with an unordered set partition $A \vdash I$ if $S = \bigsqcup_{i \in J}A_i$ for some parts $A_i$ of $A$. If $S$ is compatible with $A \vdash I$, then we let $A|S$ denote the set of parts $\{A_i\}_{i \in J}$. The \textbf{free cocommutative comonoid} on $\mbf{q}$ is the comonoid on $\mathcal{S}(\mbf{q})$ given by
\[\Delta_{S,T}((A \vdash I, \{x_i\})) = \left \{ 
 \begin{array}{ll}
 (A|S, \{x_i\}_{i \in J}) \otimes (A|T, \{x_i\}_{i \not \in J}) & \text{if $S$ is compatible with $A$} \\
 0 & \text{else}
 \end{array} \right.
 \]
 
A vector space monoid is \textbf{free commutative} if it is isomorphic to $\mathcal{S}(\mbf{q})$ as a monoid and a vector space comonoid is \textbf{cofree cocommutative} if it is isomorphic to $\mathcal{S}(\mbf{q})$ as a comonoid. 

A commutative set monoid $(\mbf{M},\Box)$ has the unique unordered factorization property if for every $x \in \mbf{M}[I]$ there is a unique expression (up to reordering of the parts) of $x$ as $x = \Box_{S_1,\cdots,S_k}(x_1,\cdots,x_k)$ where the $x_i$ are indecomposable.

As before, we have the following result.

\begin{theorem}\label{thm:cofree cocom} Let $(\mbf{C},\Delta)$ be a poset comonoid and $(\mbf{C},\Box)$ be a commutative poset monoid that form an adjoint pair with $\Delta \dashv \Box$. If $(\mbf{C},\Box)$ has the unique unordered factorization property, then $(\kk\mbf{C}, \Delta)$ is cofree cocommutative.
\end{theorem}

\begin{example} We can now prove that the Hopf monoid of graphs is cofree cocommutative. By the definition of the free product of graphs, we have
 \[ (G^c \Osq H^c)^c = G \sqcup H.\]
This means that if we decompose a graph $G$ into its connected components $G = G_1 \sqcup G_2 \sqcup \cdots \sqcup G_k$ then, 
 \[G^c = G_1^c \Osq G_2^c \Osq \cdots \Osq G_k^c.\]
Since every graph has a unique unordered decomposition into connected components, we see that graphs have the unique unordered factorization property. This tell us that $\kk\mbf{G}$ is cofree cocommutative.
\end{example}

\subsection{Fock Functor}
Crapo and Schmitt's observation concerned Hopf algebras, not Hopf monoids. We now show that the Fock functor allows us to transfer our method to Hopf algebras.

Recall that a combinatorial Hopf algebra $H$ is a connected graded vector space equipped with two linear maps $m: H \otimes H \to H$ and $\Delta: H \to H \otimes H$ which satisfy axioms similar to those of a Hopf monoid. The (first) Fock functor $\mathcal{F}$ is a functor from the category of Hopf monoids to the category of Hopf algebras given by
 \[ \mathcal{F}(\mbf{H}) = \bigoplus_{n \in \mbf{N}} \mbf{H}[ \{1, 2, \ldots, n\}]_{S_n}\]
where the subscript denotes that we are passing to the $S_n$-coinvariants where the action is the relabelling action of the species. 

If we denote the Hopf monoid by $\mbf{H}$ then we will denote the image of the Fock functor by the notation $H$. Let $[x] \in H$ denote the equivalence class of $x \in \mbf{H}[\{1,2,\cdots,n\}]$ under the relabelling action. The multiplication and comultiplication of $H$ is given by
 \[m([x],[y]) = [x \cdot y],\]
and
 \[ \Delta([x]) = \sum_{S \sqcup T = \{1,\cdots n \}} [\Delta_{S,T}(x)]\]
The naturality axiom of Hopf monoids ensures this is well-defined.

The linearized basis of $\mbf{H}[I]$ is mapped onto a basis of $H_n$ for $|I| = n$. We call this the \textbf{linearized basis} of $H$. This basis has a poset structure given by $[x] \leq [y]$ if there are representatives $x' \in [x]$ and $y' \in [y]$ such that $x' \leq y'$ in $\mbf{H}[\{1,2,\cdots,n\}]$. 

In terms of primitives, the Fock functor behaves very nicely. Recall that the primitives of a Hopf algebra $H$ are those elements $x$ such that $\Delta(x) = x\otimes 1 + 1 \otimes x$.
\begin{theorem}\label{thm:fock}
Let $H$ be a Hopf algebra that is the image of a Hopf monoid $(\mbf{H},m, \Delta)$ under the Fock functor.  Then the space of primitives of $H_n$ is the image of the space of primitives of $\mbf{H}[\{1,\cdots,n\}]$ under the relabelling action. Further if $(\mbf{H},\Box)$ is a  monoid that forms an adjoint pair $\Delta \dashv \Box$ with $(\mbf{H},\Delta)$, then the space of primitives of $H$ has a basis given by $\{ [\omega_x] \; \lvert \; \text{$x$ indecomposable with respect to $\Box$} \}$.
\end{theorem}
 
\begin{proof} Let $p \in \mbf{H}[\{1,\cdots,n\}]$, then $p$ is primitive in the Hopf monoid if and only if
 \[\Delta([p]) = \sum_{S,T} [\Delta_{S,T}(p)] = \Delta_{\{1,\cdots,n\}, \emptyset}(p) + \Delta_{\emptyset,\{1,\cdots,n\}}(p). \] 
 Then, by the counitality axiom,
 \[\Delta([p]) = [p] \otimes 1 + 1 \otimes [p].\]
Hence $p$ is primitive if and only if $[p]$ is primitive. Then this follows from Theorem \ref{thm:prim}.
\end{proof} 
 
\begin{example} The Hopf monoid of graphs $\mbf{Graphs}$ has a poset structure given by edge-inclusion. This induces a poset structure on the Hopf algebra $\mathcal{F}(\mbf{Graphs})$ of unlabelled graphs which is $[G_1] \leq [G_2]$ if there exists a labelling such that $G_1 \leq G_2$.

We have seen that the graphs that are $\Box$-indecomposable are the graphs whose complement is connected. Let $G$ be such a graph, then the following is a primitive of $\mathcal{F}(\mbf{Graphs})$
  \[ [\omega_{G}] = \sum_{G \leq G'} (-1)^{|E(G') - E(G)|} [G]. \]
where the sum is not over equivalence classes but over elements in the poset of $\mbf{G}[I]$
 \end{example}

Something curious occurs here. In order to study the primitives of a Hopf algebra using M\"obius inversion, we need use the poset of the Hopf monoid. This is
 \[ [\omega_x] = \sum_{x \leq y} \mu(x,y) [y].\]
This is \textbf{not} using the induced poset on the basis of the Hopf algebra. In the previous example, the sum
 \[\sum_{[G] \leq [H]} \mu([G],[H]) [H] \]
using the induced poset on graphs is not a primitive element. This further justifies the need for species.

\section{Self-Adjoint Hopf Monoids}

We now turn our attention to the special case where a monoid $\mbf{M}$ and a comonoid $\mbf{C}$ are compatible in the sense of Hopf algebras and also form an adjoint pair. We give a full classification of these objects; they are the linearized, commutative, and cocommutative Hopf monoids. Further, we give a description of their antipode for these objects in terms of the characteristic polynomial of an underlying poset.

\subsection{Reassembly Poset} In a paper on Hopf monoid duality, Eric Marberg constructed the following poset related to commutative and cocommutative Hopf monoids. 

\begin{theorem}\cite{Marberg}\label{thm: Marberg Poset}
Let $(\mbf{H},m,\Delta)$ be a commutative and cocommutative set Hopf monoid. Let $\leq_r$ be the binary relation given by
 \[ x\leq_r y \iff y = m_{S_1,\ldots,S_k}\circ\Delta_{S_1,\ldots,S_k}(x)\]
for $x,y \in \mbf{H}[I]$ and for some set partition $S_1 \sqcup \cdots \sqcup S_k$. Then, $\leq_r$ is a poset on $\mbf{H}[I]$.

Further, this poset has the following two properties.
\begin{enumerate}
    \item $x \leq_r m_{S,T}(y_1,y_2)$ if and only if $\Delta_{S,T}(x) = (x_1, x_2)$ for some $x_1 \leq_r y_1$ and $x_2 \leq_r y_2$.
    \item $m_{S,T}(x_1,x_2) \leq_r y$ if and only if $y = m_{S,T}(y_1,y_2)$ for some $y_1 \leq_r x_1$ and $y_2 \leq_r x_2$.
\end{enumerate}
\end{theorem}

We now show that this makes $\mbf{H}$ into a self-adjoint poset Hopf monoid. We call this the \textbf{reassembly poset} of a commutative and cocommutative Hopf monoid.

\begin{theorem} Let $(\mbf{H},m,\Delta)$ be a commutative and cocommutative set Hopf monoid. Then $\mbf{H}$ is a poset Hopf monoid with the reassembly poset structure on $\mbf{H}[I]$ given above. Further, it is self-adjoint with $\Delta \dashv m$.
\end{theorem}

\begin{proof} To prove naturality, let $\gamma: I \to I$ be a set automorphism of $I$. This induces a set map $\mbf{H}[\gamma]: \mbf{H}[I] \to \mbf{H}[I]$. To show that this map is order-preserving, let $x,y \in \mbf{H}[I]$ with $ x\leq_r y$. Then, $y = m_{S_1,\ldots, S_k} \circ\Delta_{S_1,\ldots,S_k}(x)$. Applying $\mbf{H}[\gamma]$ to both sides, we get $\mbf{H}[\gamma](y) = \mbf{H}[\gamma]\circ m_{S_1,\ldots, S_k} \circ \Delta_{S_1,\ldots,S_k}(x)$. By the naturality axiom of $m$ and of $\Delta$ this gives us
 \[ \mbf{H}[\gamma](y) = m_{\gamma(S_1),\ldots, \gamma(S_k)} \circ \Delta_{\gamma(S_1),\ldots \gamma(S_k)} \circ \mbf{H}[\gamma](x),\]
hence $\mbf{H}[\gamma](x) \leq_r \mbf{H}[\gamma](y)$.

Next, we need to show that $m$ and $\Delta$ are both order-preserving. For the multiplication, let $x_1 \leq_r y_1$ in $\mbf{H}[S]$ and $x_2 \leq_ry_2$ in $\mbf{H}[T]$. Then, for some set partitions $S_1 \sqcup \cdots \sqcup S_k = S$ and $T_1 \sqcup \cdots \sqcup T_\ell = T$, 
\begin{align*}
    y_1 &= m_{S_1,\ldots, S_k} \circ \Delta_{S_1,\ldots S_k}(x_1) \\
    y_2 &= m_{T_1, \ldots, S_\ell} \circ \Delta_{T_1, \ldots, T_\ell}(x_2)\\
\end{align*}
This means
 \[ y_1 \cdot y_2 = m_{S_1,\ldots, S_k} \circ \Delta_{S_1,\ldots S_k}(x_1) \cdot m_{T_1, \ldots, T_\ell} \circ \Delta_{T_1, \ldots, T_\ell}(x_2). \]
The compatibility axiom of a Hopf monoid allows us to reformulate this as
 \[y_1 \cdot y_2 = m_{X_1,\ldots,X_m} \circ \Delta_{X_1, \ldots, X_m}(x_1 \cdot x_2), \]
for some other set partition $X_1 \sqcup \cdots \sqcup X_m = I$. This shows that $m$ is order-preserving. A similar argument can be made for the comultiplication map $\Delta$. Therefore, $\mbf{H}$ is a poset Hopf monoid.

We now show that $\Delta \dashv m$. In other words,
\[ \Delta_{S,T}(x) \leq_r (y,z) \iff x \leq_r m_{S,T}(y,z)\]
where the poset on the left-hand-side is the product poset. This is just a rewording of property (1) in the previous theorem by Marberg.
\end{proof}

\begin{example} The set Hopf monoid of graphs $\mbf{Graphs}$ is commutative and cocommutative. This means that the previous theorem gives a poset structure on $\mbf{Graphs}[I]$ where $G \leq_r H$ if and only if $H = G|_{S_1} \sqcup \cdots \sqcup G|_{S_k}$ for some set partitions $S_1 \sqcup \cdots \sqcup S_k = I$.

This is best interpreted in terms of flats. Recall that $H \in \mbf{Graphs}[I]$ is a flat of $G \in \mbf{Graphs}[I]$ if for any edge $e \in E(G)\backslash E(H)$, the graph $E(H) \cup e$ has the same number of components as $H$. Alternatively, if $C$ is a connected component of $H$ with vertex set $S$, then $H|_S= G|_S$. From this second definition, we see that if $H$ is a flat of $G$ and $S_1,\cdots, S_k$ are the vertex sets of the connected components of $G$, then $H = G|_{S_1} \sqcup \cdots \sqcup G|_{S_k}$, so $G \leq_r H$. Conversely, if $G \leq_r H$, then $G$ is a flat of $H$. 

Geometrically, given a graph $G \in \mbf{Graphs}[I]$, the graphic hyperplane arrangement is the arrangement of hyperplanes $\{x_i = x_j\}$ for $(i,j) \in E(G)$. The intersection lattice of this arrangement is isomorphic to the poset of flats of $G$ under reverse edge inclusion. This means that if  $\emptyset$ denotes the empty graph, then the interval $[G, \emptyset]$ in the poset $\mbf{Graphs}[I]$ is isomorphic to the intersection lattice of the hyperplane arrangement of $G$. Further, the interval $[G,H]$ in the reassembly poset can be interpreted as the set of all flats of $G$ containing $H$. Thus this poset is isomorphic to the poset $[G/H,\emptyset]$ where $G/H$ is the graph obtained by contracting $H$ in $G$.
\end{example}

\subsection{Classification of Self-Adjoint Hopf Monoids}
We now prove the converse to the previous theorem. We show that every self-adjoint Hopf monoid is commutative and cocommutative.

We begin with a straight-forward extension of a result of Zelevinsky \cite{zelevinsky2006representations} on Hopf algebras to the Hopf monoid case.

\begin{lemma} Let $\mbf{H}$ be a self-adjoint poset Hopf monoid. Let $\mbf{d}$ be the subspecies of $\kk \mbf{H}$ where $\mbf{d}[I] = \{x \in \kk\mbf{H}[I] \; \lvert \; x = y \cdot z \; \text{for some y,z}\}$. Let $\mbf{p}$ be the subspecies of primitives of $\mbf{H}$. Then,
 \[\mbf{H} = \mbf{p} \oplus \mbf{d}. \]
\end{lemma}

\begin{proof} Since $\mbf{H}$ is a self-adjoint Hopf monoid, the zeta bilinear form $\langle-,-\rangle$ satisfies
 \[ \langle m_{S,T}(x,y), z\rangle = \langle x\otimes y, \Delta_{S,T}(z) \rangle. \]
We calculate the (left)-orthogonal complement to $\mbf{d}$. An element $z$ is orthogonal to $\mbf{d}$ if and only if for all $S \sqcup T$ and $x \in \mbf{H}[S]$, $y \in \mbf{H}[T]$  we have
 \[ 0 = \langle m_{S,T}(x,y), z \rangle= \langle x\otimes y, \Delta(z)\rangle. \]
By non-degeneracy of the zeta bilinear form, the right-hand side is zero for all $y \otimes z$ if and only if $\Delta_{S,T}(z) = 0$. In other words, the orthogonal complement of $\mbf{d}$ is $\mbf{p}$. We are not done since the bilinear form is non-symmetric. To finish, we need to show that the bilinear form is non-degenerate when it is restricted to $\mbf{d}^\perp = \mbf{p}$.

To do this, recall that a basis of $\mbf{p}[I]$ is given by elements of the form $\omega_{x}$ for certain $x$. Since the zeta bilinear form on these basis elements is given by $\langle \omega_x, \omega_y \rangle = \mu(x,y)$. This means that the matrix of the bilinear form restricted to $\mbf{p}$ is triangular, hence invertible. This means that $\mbf{H} = \mbf{p} \oplus \mbf{d}$.
\end{proof}

\begin{lemma} Let $\mbf{H}$ be a self-adjoint poset Hopf monoid. Then, $\kk \mbf{H}$ is commutative and cocommutative.
\end{lemma}

\begin{proof} It suffices to show that the commutator 
\[[x,y]_I = m_{S,T}(x,y) - m_{T,S}(y,x) = 0\] for all elements $x,y$ and decompositions $I = S \sqcup T$. This would imply that $\kk \mbf{H}$ is commutative and self-duality would give that it is also cocommutative. 

We proceed by induction on the size of $I$. The base case is when $I = \emptyset$. By our assumption that $\kk \mbf{H} = \kk$, we have that the bracket is zero.

Now suppose that we have proven that $[-,-]_J = 0$ for all $J$ of cardinality smaller than $n$. Let $I$ be a set of cardinality $n$. Let $\mbf{d}[I] = \kk \{x \in \kk \mbf{H}[I] \; \lvert \; x = y \cdot z\}$. By definition, $[x,y]_I$ is an element in $\mbf{d}[I]$. So, if we show that $[x,y]_I$ is also primitive, then by the previous lemma it must be $0$. Suppose $x \in {H}[S]$ and $y \in {H}[T]$.

\begin{align*}
    \Delta_{S',T'}([x,y]_I) &= \Delta_{S',T'}(m_{S,T}(x,y) - m_{T,S}(y,x)) \\
    &= [\Delta_{S \cap S', S \cap T'}(x), \Delta_{T \cap S', T \cap S'}(y)] \\
    &= [x|_{S \cap S'}, y|_{T \cap S'}] \otimes [x/_{S \cap S'}, y/_{T \cap S'}]
\end{align*}
where we explicitly use the compatibility axiom of Hopf monoids. Notice that as long as our decomposition $I = S',T'$ is non-trivial, then the inductive hypothesis gives us that this bracket is zero. Therefore, $[x,y]_I$ is in $\mbf{p}[I]$ and $\mbf{d}[I]$, which implies that it is $0$.
\end{proof}

We have now proved the following classification of self-adjoint Hopf monoids.

\begin{theorem}\label{thm: classification self-adjoint} A set Hopf monoid is self-adjoint with respect to some poset structure if and only if it is commutative and cocommutative.
\end{theorem}

\subsection{Antipode}

We now describe the antipode of a self-adjoint Hopf monoid with the reassembly poset. First, some additional properties of this poset.

\begin{lemma} Let $\mbf{H}$ be a self-adjoint Hopf monoid equipped with the reassembly poset. Let $x \in \mbf{H}[S]$ and $y \in \mbf{H}[T]$ with $I = S \sqcup T$. Then, $\omega_x \cdot \omega_y = \omega_{x\cdot y}$.
\end{lemma}

\begin{proof} The left-hand-side is
\begin{align*}
    \omega_x \cdot \omega_y &= \sum_{\substack{ x\leq_r z_1\\ y \leq_r z_2}} \mu_{\mbf{H}[S]}(x,z_1) \cdot \mu_{\mbf{H}[T]}(y,z_2) z_1 \cdot z_2 \\
    &= \sum_{(x,y) \leq_r (z_1,z_2)} \mu_{\mbf{H}[S] \times \mbf{H}[T]}((x,y),(z_1,z_2)) z_1 \cdot z_2,
\end{align*}
where $\mbf{H}[S] \times \mbf{H}[T]$ is the product poset. The right-hand-side is given by
 \[\omega_{x \cdot y} = \sum_{x\cdot y \leq_r z} \mu_{\mbf{H}[I]}(x \cdot y,z) z. \]
 
Property (2) of Theorem \ref{thm: Marberg Poset} implies that if $x \cdot y \leq_r z$, then there is a $z_1 \in \mbf{H}[S]$ and a $z_2 \in \mbf{H}[T]$ such that $(x,y) \leq_r (z_1,z_2)$. Further, it means that the intervals $[x\cdot y,z]_{\mbf{H}[I]}$ and $[(x,y),(z_1,z_2)]_{\mbf{H}[S] \times \mbf{H}[T]}$ are isomorphic and therefore have the same M\"obius number. This implies that the two sums above are the same.
 \end{proof}

Since a self-adjoint Hopf monoid is commutative and cocommutative, we have the following extension of the Cartier-Milnor-Moore Theorem to Hopf monoids first proven by Stover.

\begin{theorem}\cite{STOVER1993289}\label{thm: stover} Let $\mbf{H}$ be a commutative and cocommutative Hopf monoid and let $\mathcal{P}(\mbf{H})$ be it's species of primitives. Then,
 \[ \mbf{H} \cong \mbf{S}(\mathcal{P}(\mbf{H})).\]
In particular, $\mbf{H}$ is free commutative, free cocommutative, and self-dual.
\end{theorem}

In our case, we know that a basis of primitives is given by $\omega_x = \sum_{x\leq_r y}$ for $x$ indecomposable. Theorem \ref{thm: stover} tells us that this basis is a free generating set. This implies that every element $x \in \mbf{H}[I]$ has a unique unordered factorization $x = x_1 \cdots x_k$ into indecomposables, because if an element $x$ had two factorizations $x= x_1 \cdots x_k = x_1' \cdots x_{\ell}'$, then $\omega_x$ would have two factorizations in the free generating set. This means that reassembly poset is a graded by the function $\ell(x)$ which gives the number of terms in the unique factorization of $x$.

Recall that for a graded poset $P$ we can define the characteristic polynomial $p_{[x,y]}(t)$ of an interval $[x,y]$ as $p_{[x,y]}(t) = \sum_{x \leq z \leq y} \mu(x,z) t^{\ell(z)}$.

\begin{theorem}\label{thm:antipode} Let $\mathbf{H}$ be a self-adjoint Hopf monoid with the reassembly poset structure and let $p_{[x,y]}(t)$ denote the characteristic polynomial of the interval $[x,y]$ in $\mathbf{H}$ for some $x,y \in \mathbf{H}[I]$. Then,
 \[ S_I(x) = \sum_{x \leq_r y} p_{[x,y]}(-1) \; y.\]
\end{theorem}

\begin{proof} We prove this by computing $S_I$ on the inverted basis and then using M\"obius inversion. Towards this, let $\omega_x$ be an inverted basis element. Then,
\begin{align*}
S(\omega_x) &=\sum_{\substack{S_1 \sqcup \cdots \sqcup S_k = I\\ S_i \not = \emptyset}} (-1)^km_{S_1,\ldots,S_k} \circ \Delta_{S_1,\ldots, S_k}(\omega_x) \\
&= \sum_{\substack{S_1 \sqcup \cdots \sqcup S_k = I\\ S_i \not = \emptyset}} \sum_{\substack{x_1,\ldots, x_k \\m_{S_1,\ldots,S_k}(x_1,\ldots,x_k) = x}}(-1)^k \omega_{x_1} \cdots \omega_{x_k}
\end{align*}
The second step follows from Theorem \ref{thm: inverted basis}. The previous lemma simplifies this to the following.
\begin{align*}
S(\omega_x) &=\omega^x \sum_{\substack{S_1 \sqcup \cdots \sqcup S_k = I\\ S_i \not = \emptyset}}  \sum_{\substack{x_1,\ldots,x_k\\m_{S_1,\ldots,S_k}(x_1,\ldots,x_k) = x}} (-1)^k \\
\end{align*}

Since $x$ has a unique factorization into some irreducible elements $y_1,\ldots, y_k$, every other factorization $x = x_1 \cdot \cdot \cdot x_m$ will have the property that the $x_i$ are products of the $y_i$ such that each $y_i$ appears in exactly one $x_i$. In other words, the factorization of $x$ are in bijection with ordered set partitions of the set $\{y_1,\cdots,y_k\}$. Then, this sum is the rank-generating function of the lattice of ordered set partitions on this set evaluated at $-1$. This is
\[ S(\omega_x) = (-1)^{\ell(x)} \omega_x\]
By M\"obius inversion, we have
\begin{align*}
    S_I(x) &= \sum_{x \leq_r z} S_I(\omega_z) \\
        &= \sum_{x \leq_r z} (-1)^{\ell(y)} \omega_z
\end{align*} 
Let $\langle y \rangle \;S_I$ denote the coefficient of $y$ in the expansion of $S_I$ in the basis $\mbf{b}$. Then,
 \[ \langle y \rangle \; S_I(x) = \sum_{x \leq_r z\leq_r y} \mu(x,z)(-1)^{\ell(z)} \]
Since the poset is graded by $\ell$, this sum is exactly $p_{[x,y]}(-1)$.
\end{proof}

\begin{example} We have seen that the reassembly poset on graphs is the poset where $G \leq_r H$ if and only if $H$ is a flat of $G$. We also saw that $[G,H]$ is isomorphic to the interval $[G/H,\emptyset]$, which is isomorphic to intersection lattice of the graphical arrangement of $ G/H$. In this situation, the evaluation of the characteristic polynomial at $(-1)$ has a beautiful interpretation due to Zaslavsky. He showed that for a hyperplane arrangement $\mathcal{A}$ with intersection lattice $L(\mathcal{A})$, we have
 \[ p_{L(\mathcal{A})}(-1) = (-1)^{r(L(\mathcal{A}))} c(\mathcal{A}),\]
where $c(\mathcal{A})$ is the number of connected regions of the complement of $\mathcal{A}$ and $r(L(\mathcal{A}))$ is the rank of the poset, see for example \cite{EC1}.
For a graphical hyperplane arrangement, it is well-known that the number of connected regions in it's complement is the number of acyclic orientations. Also, the rank of the intersection lattice of $G/H$ is $|I| - rk(H)$ where the rank of $H$ is the size of the largest spanning tree pf $H$. Combining all this with the previous theorem, we obtain the following description of the antipode for $\mbf{Graphs}$.
\[S_I(G) = \sum_{\text{$H$ flat of $G$}} (-1)^{|I| - rk(H)} acyc(G/H) H,\]
where $acyc(G/H)$ is the number of acyclic orientations of $G/H$. This agrees with the previous calculations of this antipode \cite{GPHopf}.
\end{example}
\section{Examples}
In this section we give quick examples of our theory to Hopf monoids. We show how to calculate primitives and prove cofreeness for various Hopf monoids and algebras. For the self-adjoint Hopf monoids we calculate the antipode.

We have omitted the example of matroids since rereading Crapo and Schmitt's paper \cite{PrimMatroids} with our ideas in mind makes it clear how it fits into our framework.

\subsection{Set Partitions and Symmetric Functions}
A set partition of $I$ is an unordered collection of non-empty subsets $\pi = \{\pi_1, \pi_2, \cdots, \pi_k\}$ that are mutually disjoint such that $I = \pi_1 \sqcup \pi_2 \sqcup \cdots \sqcup \pi_k$. These sets $\pi_i$ are called the parts of $\pi$ and $\ell(\pi)$ is the number of blocks of $\pi$. We say that $\pi \leq \tau$ if $\tau$ refines $\pi$, that is each part of $\tau$ is fully contained in some part of $\pi$. This relation gives the set of partitions of $I$ the structure of a lattice.

If $\pi$ is a partition if $S$ and $\tau$ is a partition of $T$ with $S\cap T = \emptyset$, define their union $\pi \cup \tau$ to be the partition with parts $\{\pi_1, \cdots , \pi_k, \tau_1, \cdots, \tau_\ell\}$. If $\pi$ is a partition of $I$ and $S \subset I$, define the restriction $\pi|_S$ to be the partition $\{\pi_1 \cap S, \pi_2 \cap S, \cdots, \pi_k \cap S\}$ where we remove the empty sets that appear in these intersections.

\begin{definition} The \textbf{poset Hopf monoid of set partitions} is the Hopf monoid on the poset species
 \[ \mbf{Par}[I] = \text{Lattice of partitions of $I$}\]
where with structure maps
 \[ m_{S,T}(\pi,\tau) = \pi \cup \tau \quad \quad \text{and} \quad \quad \Delta_{S,T}(\pi) = (\pi|_S, \pi|_T)\]
\end{definition}

It is not hard to see that 
    \[(\pi|_S, \pi|_T) \leq (\tau_1, \tau_2) \iff \pi \leq \tau_1 \cup \tau_2. \]
So partitions forms a self-adjoint Hopf monoid. This immediately tells us that this Hopf monoid is free commutative, free cocommutative self-dual, by Theorem \ref{thm: stover}. We can also calculate the primitives. Let $\pi_I$ denote the partition with one part $I$, which is the only indecomposable of $\mbf{Par}[I]$. Then the primitives are scalings of

\[ \omega_{\pi_I} = \sum_{\tau \in \mbf{Par}[I]} \mu(\pi_I,\tau) \tau. \]
More explicitly,
 \[ \omega_{\pi_I} = \sum_{\tau \in \mbf{Par}[I]} \left (\prod_{\tau_i} (-1)^{|\tau_i| - 1}(|\tau_i| - 1)!\right ) \tau = \sum_{\tau \in \mbf{Par}[I]} (-1)^{|I| - \ell(\tau)}\left (\prod_{\tau_i} (|\tau_i| - 1)!\right ) \tau.\]

Applying the first Fock functor to this Hopf monoid gives us a Hopf algebra of set partitions up to relabelling. This is the same as the Hopf algebra of integer partitions $Par = \bigoplus_{n}Par_n$ where the vector space $Par_n$ is spanned by the integer partitions of $n$. The poset structure induced on the basis of integer partitions is the refinement of integer partitions. For a given set partition $\pi$, let $\lambda(\pi)$ denote its underlying integer partition which is given by the size of the parts. By our results in the Fock functor section, a basis of the primitives of $Par$ is given by
\[ \omega_{(n)} = \sum_{\tau \in \mbf{Par}[I]} \mu(\pi_I,\tau) \lambda(\tau) .\]

A result of Marcelo Aguiar and Federico Ardila \cite{GPHopf} is that we have an isomorphism between the Hopf algebra of partitions $Par$ and the Hopf algebra of symmetric functions $\Lambda$ given by
 \[ n! \; (n) \mapsto h_n,\]
where $(n)$ is the integer partition of $n$ with one part and $h_n$ is the homogeneous basis of $\Lambda$.

We can combine this isomorphism with our description of primitives of Theorem \ref{thm:prim} to get a description of the primitives of $\Lambda$.  We get a basis of the primitives of $\Lambda$ given by
    \[ \omega_{h_n} = \sum_{\tau \in \mbf{Par}[{1,2,\cdots,n}]} \mu(\pi_i, \tau)
\]
Since this a graded basis for the space of primitives, standard symmetric function theory tells us that these $\omega_{h_n}$ coincide with the power sums $p_n$ up to a factor.
This recovers a classical result by Doubilet \cite{Doubilet} that expresses the power sum functions in terms of M\"obius inversion of the homogeneous basis over the set partition lattice. 
\begin{theorem}\cite{Doubilet} The power sum symmetric functions $p_n \in \Lambda$ are given by
\[ p_n = \frac{1}{\mu(\hat{0}, \hat{1})} \sum_{\Phi} \mu(\hat{0},\Phi) h_{\lambda(\Phi)},\]
where the sum is over the set partition lattice.
\end{theorem}

Notice that we did not have to guess that the lattice of set partitions was the important lattice for this result. Everything came from general machinery. 

Using Theorem \ref{thm:antipode}, we can also calculate the antipode of partitions. First, notice that the reassembly poset $\leq_r$ of partitions is the same as the poset above. Namely, $\pi \leq_r \tau$ if $\tau$ refines $\pi$.  Then, the antipode is
\[ S_I(\pi) = \sum_{ \pi \leq_r \tau} p_{[\pi , \tau ]}(-1) \tau\]

The characteristic polynomial of the set partition lattice is well-studied. A calculation is given in Example 3.10.4 in Stanely's book \cite{EC1}. He shows,
 \[ \chi_{\mbf{Par}[I]}(t) = t(t-1)(t-2)\cdot \cdot \cdot(t-n+1).\]
Which means,
 \[\chi_{\mbf{Par}[I](-1)} = (-1)^{n} n!\]
Further, we know that any interval in the set partition lattice $[ \pi, \tau]$ is a product of smaller set partition lattices. Suppose that $\pi = \{B_1,\cdots,B_{\ell(\pi)}\}$ and that in the refinement $\tau$ each block $B_i$ is partitioned into $\lambda_i$ blocks. If $\Pi_{\lambda_i}$ denotes the set partition lattice of a set of size $\lambda_i$. Then,
 \[[\pi, \tau] \cong \Pi_{\lambda_1} \times \cdots \times \Pi_{\lambda_{\ell(\pi)}} \]

Combining these two facts alongside the realization that $\lambda_1 + \cdots + \lambda_{\ell(\pi)} = \ell(\tau)$ shows that 
 \[p_{[\pi , \tau ]}(-1) = (-1)^{\ell(\tau)}\prod_{B_i \in \pi} \lambda_i!\;.\] 
 Thus, we have
 \[ S(\pi) = \sum_{\pi \leq_r \tau} (-1)^{\ell(\tau)} \left (\prod_{B_i \in \pi}\lambda_i! \right ) \tau.\]
This result agrees with Aguiar and Ardila's calculation in \cite{GPHopf}.

\subsection{Hypergraphs and Simplicial Complexes}
Throughout the examples, we have already encountered the Hopf monoid of graphs. We calculated its subspecies of primitives, proved that it is cofree, and proved self-duality. We now extend these results to hypergraphs and to the submonoid of simplicial complexes.

A \textbf{hypergraph} $\mathcal{G}$ on vertex set $I$ is the tuple $(I, E)$ where $E \subset 2^I$ that does not contain the empty set or any singleton. The elements in $E$ are called hyperedges.We can define a species by
 \[ \mbf{HG}[I] = \{ \mathcal{G} \; \lvert \; \text{$\mathcal{G}$ is a hypergraph on vertex set $I$} \}.\]
 Define a poset structure by $\mathcal{H} \leq \mathcal{G}$ if and only if $E(\mathcal{H}) \subset E(\mathcal{G})$. This makes $\mbf{HG}$ a poset species. If $\mathcal{G}_1$ is a hypergraph on $S$ and $\mathcal{G}_2$ is a hypergraph on $T$, define their product by
 \[m_{S,T}(\mc{G}_1, \mc{G}_2) = (S \sqcup T, E(\mc{G}_1) \cup E(\mc{G}_2)). \]
If $\mc{G}$ is a hypergraph on $I$ and $S \subset I$, then the restriction of $\mc{G}$ to $S$, denoted by $\mc{G}|S$ is the hypergraph
 \[ \mc{G}|S = \left (S, \{ U \in E(\mc{G}) \; \lvert \; U \subset 2^S\} \right).\]
 
Then

\begin{lemma} The poset species $\mbf{HG}$ equipped with the multiplication $m_{S,T}$ above and comultiplication $\Delta_{S,T}: \mc{G} \mapsto (\mc{G}|S, \mc{G}|T)$ is a poset Hopf monoid.
\end{lemma}

\begin{proof} It is proved in \cite{Simp} that this forms a Hopf monoid. The only remaining work is to show that the multiplication and comultiplication maps are order-preserving. Since disjoint union and restriction are both order-preserving operations, the lemma follows.
\end{proof}

This Hopf monoid has been studied in \cite{Simp} and \cite{BBM}. Note that this Hopf monoid is different from the Hopf monoid of hypergraphic polytopes studied by Aguiar and Ardila in \cite{GPHopf} where they allow hypergraphs to have repeated hyperedges.

Let $\mc{K}(S,T)$ be the complete bipartite hypergraph on vertex sets $S$ and $T$. This graph contains all hyperedges which are not strictly contained in $2^S$ and $2^T$. Define the free product on hypergraphs by
 \[ \mc{G}_1 \Osq \mc{G}_2 = \left (I, E(\mc{G}_1) \sqcup E(\mc{G}_2) \cup E(\mc{K}(S,T)) \right). \]
Just as in the case of graphs, we have

\begin{lemma} The poset comonoid $(\mbf{HG},\Delta)$ and the poset monoid $(\mbf{HG},\Box)$ form an adjoint pair with $\Delta_{S,T} \dashv \Box_{S,T}$.
\end{lemma}

Given a hypergraph $\mc{G} = (I, E)$, the \textbf{complement} of $\mc{G}$ is given by
 \[ \mc{G}^c = (I, \{ U \subset 2^I \; \lvert \; U \not \in E(\mc{G}) \text{ and $|U| \geq 2$} \})\]
A hypergraph is connected if it is not the disjoint union of two other hypergraphs. As with graphs, we have
 \[ (\mc{G}_1^c \Osq \mc{G}_2^c)^c = \mc{G}_1 \sqcup \mc{G}_2.\]
This means that if $\mc{G}$ is the complement of a connected graph, then it is $\Osq$-indecomposable. Theorem \ref{thm:prim} and Theorem \ref{thm:fock} imply the following theorem.

\begin{proposition} The primitive species $\mathcal{P}(\kk \mbf{HG})[I]$ of hypergraphs has a basis given by 
 \[ \omega_{\mc{G}} = \sum_{\mc{G}\leq \mc{H}} \mu(\mc{G},\mc{H}) \mc{H} ,\]
for hypergraphs $\mc{G}$ on $I$ whose complement is connected. 

Further, the subspace of primitives of the Hopf algebra of hypergraphs $\mathcal{F}(\kk \mbf{HG})$ has a basis given by
 \[ [\omega_{\mc{G}} ] = \sum_{\mc{G} \leq \mc{H}} \mu(\mc{G},\mc{H}) [\mc{H}],\]
for hypergraphs $[\mathcal{G}]$ whose complement is connected.
\end{proposition}

 Moreover, if $\mc{G}_1 \in \mbf{HG}[S], \mc{G}_2 \in \mbf{HG}[T]$ and $\mc{H} \in \mbf{HG}[I]$ with $I = S \sqcup T$, then we also have
 \[ m_{S,T}(\mc{G}_1, \mc{G}_2) \leq \mc{H} \iff (\mc{G}_1, \mc{G}_2) \leq (\mc{H}|S, \mc{H}|T).\]
This means that $m_{S,T} \dashv \Delta_{S,T}$. So, $\mbf{HG}$ is self-adjoint and hence commutative and cocommutative. Theorem \ref{thm: stover} implies the following result.

\begin{proposition} The Hopf monoid of hypergraphs $\kk \mbf{HG}$ is free commutative, free cocommutative, and self-dual.
\end{proposition}

Since this Hopf monoid is self-adjoint, Theorem \ref{thm:antipode} gives the following grouping-free formula for this Hopf monoid.

\begin{theorem} Let $\leq_r$ be the reassembly poset on the Hopf monoid of hypergraphs $\mbf{HG}$. Let $p_{[\mathcal{H},\mathcal{G}]}$ denote the characteristic polynomial of the interval $[\mathcal{H},\mathcal{G}]$ in reassembly poset. Then,
\[ S_I(\mathcal{G}) = \sum_{\mathcal{G} \leq_r \mathcal{H}} p_{[\mathcal{G},\mathcal{H}]}(-1) \; \mathcal {H}.\]
 
 \end{theorem}

This gives a new interpretation for the antipode of this Hopf monoid given by Benedetti, Bergeron, and Machacek \cite{BBM}. Notice that a hypergraph $\mathcal{G}$ is less than $\mathcal{G}$ in the reassembly poset if and only if there exists an ordered set partition $S_1 \sqcup \cdots \sqcup S_k$ of $I$ such that $\mathcal{H} = \prod \mathcal{G}|_{S_i}$. In the language of \cite{BBM}, this means that $\mathcal{H}$ is a flat of $\mathcal{G}$.

One commonly studied submonoid of $\mbf{HG}$ is the Hopf monoid of simplicial complexes defined in \cite{Simp}. A \textbf{simplicial complex} on ground set $I$ is a collection of subsets $\Gamma \subset 2^I$ that is downward closed; that is, if $Y \subset X \subset I$ and $X \in \Gamma$, then $Y \in \Gamma$. If $S \subset I$, then we can define the restriction of $\Gamma$ to $S$ by $\Gamma|S = \Gamma \cap 2^S$. Given a simplicial complex $\Gamma_1$ on ground set $S$ and a simplicial complex $\Gamma_2$ on ground set $T$, define their disjoint union by $\Gamma_1 \sqcup \Gamma_2 = \{A \subset I \; \lvert \; A \in \Gamma_1 \text{ or } A \in \Gamma_2$\}. This defines a Hopf monoid $\mbf{SC}[I] = \{ \text{simplicial complexes with label set $I$}\}$ with structure maps
 \[m_{S,T}(\Gamma_1,\Gamma_2) = \Gamma_1 \sqcup \Gamma_2\;, \]
and
 \[\Delta_{S,T}(\Gamma) = (\Gamma|S, \Gamma|T). \]
This comes with a natural poset structure where $\Gamma \leq \Phi$ if and only if every $X \in \Gamma$ is in $\Phi$. 

Since a simplicial complex can be interpreted as a special type of hypergraph, we see that $\mbf{SC}$ embeds as a Hopf monoid into $\mbf{HG}$. We have to be careful since this does not immediately imply that every result descends. To see this, notice that the free product of two simplicial complexes (viewed as hypergraphs) is no longer a simplicial complex. However, we still have that $m_{S,T} \dashv \Delta_{S,T}$. That is
 \[\Gamma_1 \sqcup \Gamma_2 \leq \Phi \iff (\Gamma_1, \Gamma_2) \leq (\Phi|S, \Phi|T). \]
This means that $\mbf{SC}$ is self-adjoint. Thus, we have the following result.

\begin{theorem} The Hopf monoid of simplicial complexes $\kk \mbf{SC}$ is free commutative, free cocommutative, and self-dual.
\end{theorem}

Further, we calculate the space of primitives. The simplicial complexes that are $m$-indecomposable are those that can not be written as a disjoint union of other simplicial complexes. These are called the connected simplicial complexes. Since $m_{S,T} \dashv \Delta_{S,T}$, we need to reverse the poset order to apply Theorem \ref{thm:prim}. This gives

\begin{proposition} The space of primitives $\mathcal{P}(\kk \mbf{SC})[I]$ has a basis given by 
 \[ \omega^{\Gamma} = \sum_{\Phi \leq \Gamma} \mu(\Phi, \Gamma) \Gamma,\]
for connected $\Gamma \in \mbf{SC}[I]$.
\end{proposition}

The reassembly poset $\leq_r$ on $\mbf{SC}$ is easier to describe that that of $\mbf{HG}$. For a simplicial complex $\Gamma$, let $\Gamma^{(1)}$ denote its 1-skeleton; that is, the set of subsets of size 2 in $\Gamma$. This forms a graph. For any flat $F \subset \Gamma^{(1)}$, let $C_1 \sqcup \cdots \sqcup C_k$ denote the connected components of the flat. Then, define $\Gamma(F) = \Gamma|_{C_1} \sqcup \cdots \sqcup \Gamma|_{C_k}$.

\begin{lemma} In the reassembly poset, $\Gamma \leq_r \Phi$ if and only if $\Phi = \Gamma(F)$ for some flat $F \subset \Gamma^{(1)}$. Therefore, the interval $[\Gamma,\emptyset]$ in the reassembly poset is isomorphic to the interval $[\Gamma^{(1)},\emptyset]$ in the reassembly poset of graphs.
\end{lemma}

\begin{proof} First, if $\Phi = \Gamma(F)$, then using the set partition $S_1 \sqcup \cdots \sqcup S_k$ induced by the connected components $C_1 \sqcup \cdots \sqcup C_k$ of $F$, we clearly see that $\Gamma(F) = \Gamma|_{C_1} \sqcup \cdots \sqcup \Gamma|_{C_k} = m_{S_1,\cdots,S_k}\circ \Delta_{S_1,\cdots,S_k}(\Gamma)$.

In the other direction, for any set partition $S_1 \sqcup \cdots \sqcup S_k = I$, we need to show that $(\Gamma|_{S_1})^{(1)} \sqcup \cdots \sqcup (\Gamma|_{S_k})^{(1)}$ is a flat of $\Gamma^{(1)}$. Since simplical complexes are downward closed, we have that $\Gamma|_{S_1}^{(1)} = \Gamma^{(1)}|S_1$. Therefore, $(\Gamma|_{S_1})^{(1)} \sqcup \cdots \sqcup (\Gamma|_{S_k})^{(1)} = \Gamma^{(1)}_{S_1} \sqcup \cdots \sqcup \Gamma^{(1)}_{S_k}$. Thus this is indeed a flat of $\Gamma^{(1)}$.
\end{proof}

The previous Lemma combined with our Theorem \ref{thm:antipode} gives us a new proof of the following result by Benedetti, Hallam, and Machacek \cite{Simp}.

\begin{theorem}\cite{Simp}
Let $acyc(G)$ denote the number of acyclic orientations of a graph. Then, the antipode for $\mbf{SC}$ is given by
 \[S_I(\Gamma) = \sum_{\substack{F \subset \Gamma^{(1)}\\ \text{flat}}} (-1)^{|I| - rk(F)} acyc(\Gamma^{(1)}/F) \; \Gamma(F). \]

\end{theorem}


\bibliography{main.bib}
\bibliographystyle{hieeetr.bst}

\end{document}